\documentclass[11pt]{article}

\usepackage{array}
\usepackage[T1]{fontenc}
\usepackage{graphicx}
\usepackage{amsfonts}
\usepackage{amsmath}
\usepackage{amsfonts}
\usepackage{amsthm}
\usepackage{amssymb}
\usepackage{amscd}
\usepackage{epsfig}
\usepackage{verbatim}
\usepackage{fancybox}
\usepackage{moreverb}
\usepackage{psfrag}
\usepackage{cellspace}
\usepackage{hyperref}
\usepackage{latexsym}
\usepackage{psfrag}
\usepackage{color}

\numberwithin{equation}{section}
\definecolor{blanc}{rgb}{1.,1.,1.}
\definecolor{vert}{rgb}{0.,0.5,0.}
\definecolor{rouge}{rgb}{0.8,0.,0.}
\definecolor{violet}{rgb}{0.5,0.,0.4}
\definecolor{bleu}{rgb}{0.,0.,0.5}
\definecolor{orange}{rgb}{0.8,0.4,0.}
\definecolor{light-blue}{rgb}{0.5,0.5,0.7}
\definecolor{light-red}{rgb}{0.8,0.4,0.4}
\definecolor{noir}{rgb}{0.,0.,0.}
\definecolor{gris}{rgb}{0.8,0.8,0.7}

\newcommand{\eps}{\varepsilon}

\newcommand{\yv}{\boldsymbol{y}}
\newcommand{\rv}{\boldsymbol{r}}

\newcommand{\inp}{i\in\{1,\dots,N_p\}}

\newtheorem{theorem}{Theorem}[section]
\newtheorem{lemma}[theorem]{Lemma}

\newtheorem{remark}[theorem]{Remark}

\newtheorem{algorithm}[theorem]{Algorithm}
\newtheorem{approximation}[theorem]{Approximation}

\def\eps{\varepsilon}

\begin{document}

\title{Long time simulation of a highly oscillatory Vlasov equation with an
exponential integrator}
\author{
Emmanuel Fr\'enod\thanks{Université Bretagne-Sud, UMR 6205, LMBA, F-56000 Vannes,
France 
\& Inria Nancy-Grand Est, TONUS Project }
\and
Sever A. Hirstoaga\thanks{Inria Nancy-Grand Est, TONUS Project \& IRMA
(UMR CNRS 7501), Universit\'e de Strasbourg, France}
\and
Mathieu Lutz \thanks{IRMA (UMR CNRS 7501), Universit\'e de Strasbourg,
7 rue Ren\'e Descartes, F-67084 Strasbourg Cedex, France \&
Inria Nancy-Grand Est, TONUS Project}
}
\date{}
\maketitle

\begin{abstract}
We change a previous time-stepping algorithm for solving a multi-scale
Vlasov-Poisson system within a Particle-In-Cell method, in order to do accurate
long time simulations. As an exponential integrator, the new scheme allows to
use large time steps compared to the size of oscillations in the solution.
\end{abstract}



\section{Introduction}

In this paper, we improve the time numerical scheme introduced in
\cite{frenod/hirstoaga/sonnendrucker}, in order to make long time simulation of
the following Vlasov equation
\begin{gather}
\label{VlasAxSymBeam}   
\left\{
\begin{aligned}
    &\frac{\partial f_{\eps}}{\partial t}+\frac{v}{\eps}\frac{\partial
    f_{\eps}}{\partial r}+\Big(-\frac{r}{\eps}+E_{\eps}(r,t)\Big)\frac{\partial
    f_{\eps}}{\partial v}=0,
    \\
    &f_{\eps}\negthickspace\left(0,r,v\right)=f_{0}\left(r,v\right),
\end{aligned}
\right.
\end{gather}
where $0<\eps\ll1$, $r\in\mathbb{R}$ is the position variable, $v\in\mathbb{R}$
the velocity variable, $f_{\eps}=f_{\eps}(r,v,t)$ is the distribution
function, $f_0$ is given, and $-\frac{r}{\eps}+E_{\eps}(r,t)$ corresponds to the
electric field. The difficulty in the numerical solution of this equation lies
mainly in the multi-scale aspect due to the small parameter $\eps$.

The main application will be the case when the electric field $E_{\eps}$ is 
obtained by solving the Poisson equation. We will thus have to solve the
following nonlinear system of equations
\begin{gather}
\label{VlasPoissAxSymBeam}   
\left\{
\begin{aligned}
  &\frac{\partial f_{\eps}}{\partial t}+\frac{v}{\eps}\frac{\partial
    f_{\eps}}{\partial r}+\big(-\frac{r}{\eps}+E_{\eps}\left(r,t\right)\big)
  \frac{\partial f_{\eps}}{\partial v}=0,  \\
     &\frac{\partial}{\partial r}\big(rE_{\eps}\big)=\rho_{\eps}\negmedspace
  \left(t,r\right), \hspace{3mm} \rho_{\eps}\negmedspace\left(t,r\right)=
  \int_{\mathbb{R}}f_{\eps}\left(t,r,v\right)dv,   \\  
     &f_{\eps}\negthickspace\left(0,r,v\right)=f_{0}\left(r,v\right).
\end{aligned}
\right.
\end{gather}
This simplified model of the full Vlasov-Maxwell system is particularly adapted
to the study of long and thin beams of charged particles within the paraxial
approximation assuming that the beam is axisymmetric (see
\cite{2007arXiv0710.3983F} and the references therein).

We will also test our scheme when 
\begin{equation} \label{RappelForce}
  E_{\eps}(r,t)=-r^3.
\end{equation}

As in \cite{frenod/hirstoaga/sonnendrucker}, we solve numerically the Vlasov
equation by a particle method, which consists in approximating the distribution
function by a finite number of macroparticles (see \cite{Birdsall}). The
trajectories of these particles are computed from the characteristic curves of
the Vlasov equation, whereas the term $E_{\eps}$ is computed, when obtained from
the Poisson equation, on a mesh in the physical space. The contribution of this
work is to improve the time numerical scheme proposed in
\cite{frenod/hirstoaga/sonnendrucker}, for solving in long times the stiff
characteristics associated to equation \eqref{VlasAxSymBeam}
\begin{gather}
\label{CharactCurvesOfGenVlas}  
\left\{
\begin{aligned}
     &\frac{\text{d} R}{\text{d} t}=\frac{V}{\eps},
     \\
     &\frac{\text{d} V}{\text{d} t}=-\frac{R}{\eps}+E_{\eps}(R,t),
\end{aligned}
\right.
\end{gather}
provided with the initial conditions $R(0)=r$, $V(0)=v$. We note that when
$E_{\eps}$ is given by \eqref{RappelForce}, solving \eqref{CharactCurvesOfGenVlas}
reduces to solve the undamped and undriven Duffing equation
\begin{equation}\label{DufDuf}  
R''\left(t\right)+\frac{R(t)}{\eps^2}+\frac{R^{3}(t)}{\eps}=0
\end{equation}
with initial conditions $R(0)=r$, $R'(0)=v/\eps$.

The solution of \eqref{VlasAxSymBeam} is represented by a beam of particles in
phase space moving along the characteristics curves \eqref{CharactCurvesOfGenVlas}.
When the electric field $E_{\eps}$ is zero the beam rotate around the origin with
period $2\pi\eps$. Otherwise, the dynamical system \eqref{CharactCurvesOfGenVlas}
can be viewed as a perturbation of the system obtained when the electric field is
zero (see \cite{lutz}). When the electric field $E_{\eps}$ is obtained by solving
the Poisson equation or if it is given by \eqref{RappelForce}, the beam evolves
by  rotating around the origin in the phase space, and in long times a bunch forms
around the center of the beam from which  filaments of particles are going out.
Classical explicit numerical methods require very small time step in order to
capture the fast rotation. Consequently, long time simulations are very tedious.

In \cite{frenod/hirstoaga/sonnendrucker} the authors introduced a new explicit
scheme based on Exponential Time Differencing (ETD) integrators (see 
\cite{cox/matthews} and \cite{HochbruckOstermann}). The idea of the algorithm
was the following: first, for each macroparticle as initial condition, we solve
the ODE in \eqref{CharactCurvesOfGenVlas} over some fast ({\em i.e.} of order
$\eps$) time by using a 4th order Runge-Kutta solver. The fast time is the same 
for all the macroparticles. Then, by means of the variation-of-constants formula
we compute an approximation of the solution over
a large whole number of fast times. Thus, this scheme (recalled in Algorithm
\ref{ClassicETDAlgo}) allows to take time steps $\Delta t$ varying from $0.1$ to
$1$ and therefore, when $\eps$ varies from $10^{-5}$ to $10^{-2}$, the numerical
gain increases from $3$ to $10^5$ when compare the scheme to a reference
solution. In addition, for short time simulation ($t=3.5$) the results are quite
precise (the numerical error is about $10^{-3}$ knowing that the beam is of size
$1$). For time of simulation less than $30$, this ETD method still gives good
qualitative behaviors. Nevertheless, it fails to capture precisely the
filamentation phenomena leading even to unstable simulation in long times.

In this paper we improve the time-stepping method following a simple idea.
The fast time in the above description of the algorithm was chosen
in \cite{frenod/hirstoaga/sonnendrucker} to be the mean of all the numerically
computed periods for the particle trajectories. The authors remarked that this
choice gives better results than using $2\pi\eps$. The algorithm that we propose
now is to not use a fixed value for the fast time, but to push each particle
with its period. This means that the first step should consist in solving the ODE
starting from each particle over different fast times (the particles computed
periods). However, implementing this simple idea turns out to be not so easy and
in addition asks for changing the order of the steps in the algorithm. We detail
this problem in Section \ref{Section3}. Our simulations show that using precise
periods for each particle and at each macroscopic time step results in a
more accurate ETD scheme in long times. In addition, by following the proof lines
in \cite{frenod/hirstoaga/sonnendrucker}, we show formally that the algorithm is
still an approximation when using any value for the fast time. We recall that in
\cite{frenod/hirstoaga/sonnendrucker}, the authors justified the algorithm when
using only $2\pi\eps$ as fast time. From now on we will use the incorrect term
of period to indicate the time that a particle takes to do a complete tour in
the phase space. 

The remainder of the paper is organized as follows. In Section \ref{Section2} we
briefly recall the ETD scheme introduced in \cite{frenod/hirstoaga/sonnendrucker}
and we do long time simulation. Then, Section \ref{Section3} is devoted to the
improvement of the method. Eventually, in Section \ref{Section4}, we implement 
and test the new algorithm on the Vlasov equation in the two cases of electric
field.



\section{Long time simulation of the exponential integrator}
\label{Section2}  

We first recall the exponential integrator introduced in
\cite{frenod/hirstoaga/sonnendrucker} for solving efficiently in short times the
Vlasov equation. Then, we discuss the long time simulations done with this
algorithm in the case of $E_\eps$ given by \eqref{RappelForce}, in Section
\ref{sec2.1} and in the Vlasov-Poisson case in Section \ref{sec2.3}.

\subsection{The previous exponential integrator for the Particle-In-Cell method}
The time-stepping scheme that we describe in this section for solving the Vlasov
equation \eqref{VlasAxSymBeam} is proposed in the framework of a Particle-In-Cell
method (PIC).  A PIC method (see \cite{Birdsall}) consists in approximating the
distribution function $f_\eps$, solution of \eqref{VlasAxSymBeam}, by a cloud of
macroparticles  $(R_i,V_i)_{\inp}$ (each one of them stands for a set of
physical particles) of weight $\omega_i$ advanced along the characteristics
curves \eqref{CharactCurvesOfGenVlas}. In other words we will approach the
distribution function by a sum of Dirac masses centered at the macroparticles
positions and velocities
\begin{align}
  &f_{\eps}\negmedspace\left(r,v,t\right)\approx\overset{N_p}{\underset{i=1}{\sum}}
  \omega_i\delta\left(r-R_i(t)\right)\delta\left(v-V_i(t)\right).
\end{align}
The weights are chosen such that 
\begin{align}
    &\overset{N_p}{\underset{i=1}{\sum}}\omega_i=\int_{\mathbb{R}^{2}}f_{0}(r,v)\,drdv,
\end{align}
and the particles are initialized according to the probability density associated
with $f_0$.

The aim of the exponential integrator introduced in
\cite{frenod/hirstoaga/sonnendrucker} is to solve \eqref{CharactCurvesOfGenVlas}
by using big time steps with respect to the typical period of oscillation, which
is about $2\pi\eps$. Now, we recall the main ideas of this algorithm. 
For simplicity we denote by $\boldsymbol{y}$ the characteristics, \textit{i.e.}
\begin{align}
  &\boldsymbol{y}=\left(\begin{array}{c}
    R\\
    V\end{array}\right).
    \label{SimplNotations}  
\end{align}
Under the assumption
\begin{equation}\label{assu}
\text{for all } n\;\;  \int_{t_n}^{t_n +N\cdot 2\pi\eps}
      {\cal R}(\tau) \left(\begin{array}{c}
        0 \\
        E_{\eps}(\tau, R(\tau))
      \end{array}\right)d\tau \approx N\cdot \int_{t_n}^{t_n+2\pi\eps}
      {\cal R}(\tau)\left(\begin{array}{c}
        0 \\
        E_{\eps}(\tau, R(\tau)) \end{array}\right)d\tau,
\end{equation}
where ${\cal R}$ is some $2\pi\eps$-periodic rotation in the phase space, the
authors derived the following scheme. Fix a time step $\Delta t\gg2\pi
\eps$ and determine the unique integer $N$ and the unique real $o$ such that 
\begin{gather}
\left\{
\begin{aligned}
    &\Delta t=N\cdot 2\pi\eps+o,
    \\
    &0\leq o< 2\pi\eps.
\end{aligned}
\right.
\end{gather}
For each macroparticle $\yv_i$ (denoted by $\yv$ for simplicity) apply
\begin{algorithm}
\label{ClassicETDAlgo}
Assume that $\boldsymbol{y}_{n}$ the solution of \eqref{CharactCurvesOfGenVlas}
at time $t_n$ is given.
\begin{enumerate}
\item Compute $\boldsymbol{y}\left(t_{n}+2\pi\eps\right)$ by using a fine
Runge-Kutta solver with initial condition $\yv_n.$
\item Compute $\boldsymbol{y}\left(t_{n}+N\cdot 2\pi\eps \right)$ by the rule
\begin{align}\label{step2ClassicalETDAlgo}
    &\boldsymbol{y}\left(t_{n}+N\cdot 2\pi\eps \right)=\boldsymbol{y}_{n} + 
  N\big(\boldsymbol{y}\left(t_{n}+2\pi\eps\right)-\boldsymbol{y}_{n}\big)
\end{align}
\item Compute $\boldsymbol{y}$ at time $t_{n+1}$ by using a fine Runge-Kutta
solver with initial condition $\yv( t_{n}+N\cdot 2\pi\eps )$, obtained
at the previous step.
\end{enumerate}
\end{algorithm}
\begin{figure}[ht]
\begin{center}
  \includegraphics[scale=0.5]{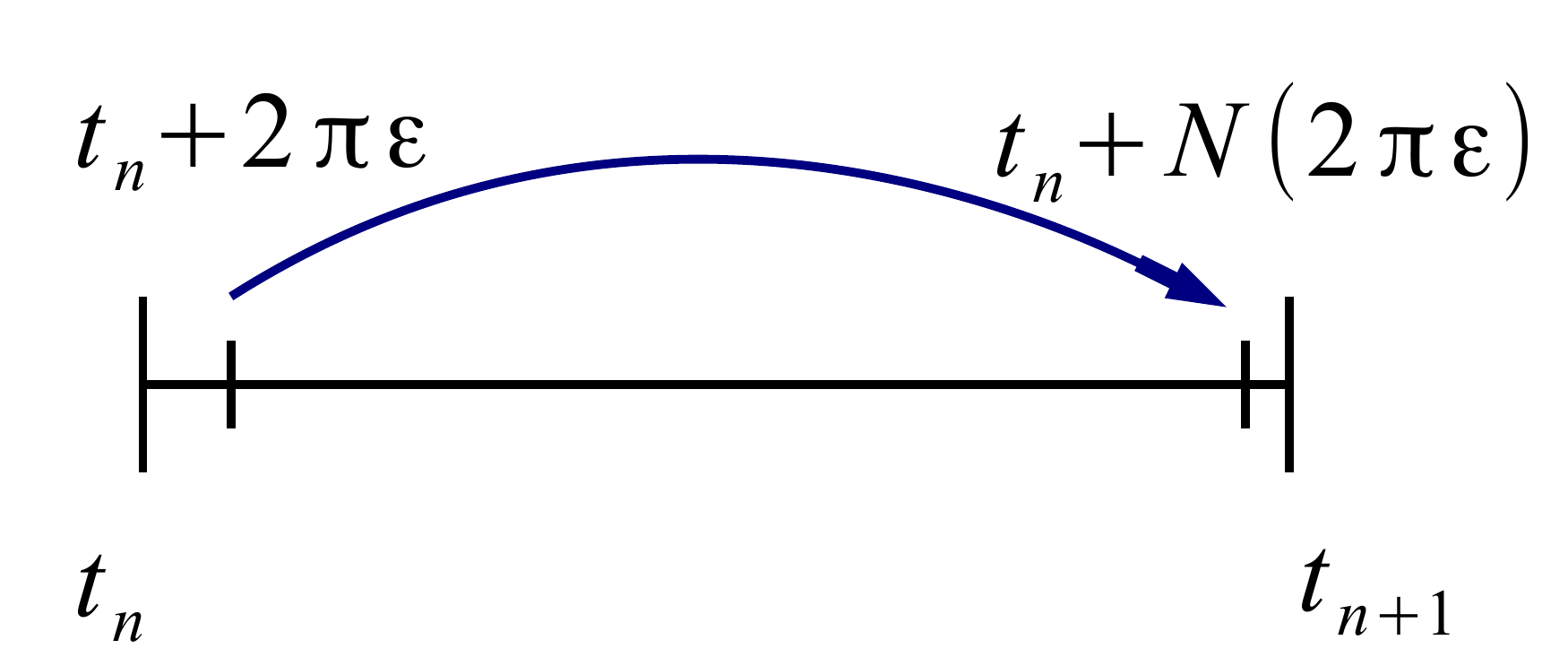}
\end{center}
\label{SchematicClassicalETDScheme}   
\end{figure}
\begin{remark}\vspace*{-6mm}
Hypothesis \eqref{assu} is satisfied if the time to make one
rapid complete tour (in the phase space) is quasi constant and close to $2\pi
\eps$ and if $\tau\mapsto E_{\eps}(\tau,\boldsymbol{y}(\tau))$ is about
$2\pi\eps$-periodic.
\end{remark}

This simple algorithm was formally proved in \cite{frenod/hirstoaga/sonnendrucker}
to provide an approximation of the solution under the assumption in \eqref{assu}.
The capability of the scheme relies on the following. When the solution is periodic,
say $\yv(t_n+2\pi\eps)=\yv_n$, the second step of the algorithm reduces to an
exact statement. Otherwise, it can happen that $\left\Vert\yv(t)\right\Vert_2=
\sqrt{R(t)^{2}+V(t)^{2}}$ vary slowly in time, typically when the electric field is
given by the Poisson equation (see \cite{lutz}). In this case, the scheme allows to
capture the small variations of the amplitude of $\yv$ since it solves them during
the first step. Nevertheless, it should be noted that the algorithm works provided
that these variations do not change considerably during the second step.

Therefore, our aim now is to improve the algorithm in order to be able to capture
also the small variations of the period. This viewpoint was outlined in
\cite{frenod/hirstoaga/sonnendrucker}. More precisely, the authors noticed that
the results are much better, especially in the Vlasov-Poisson case, when
replacing $2\pi\eps$ in Algorithm \ref{ClassicETDAlgo} by the mean of the
particles times to make the first complete tour. The so-obtained scheme was
called the {\em modified ETD algorithm}. However, as we will see in the next two
sections, this scheme still needs to be improved in order to do accurate long
time simulations. The idea is that using a more accurate period when pushing
particles leads to better results.

\subsection{Application to long time simulation of the undamped and undriven
Duffing equation} \label{sec2.1}
\begin{figure}[ht]
\begin{center}\hspace*{-8mm}
\begin{tabular}{cc}
  \includegraphics[scale=0.3]{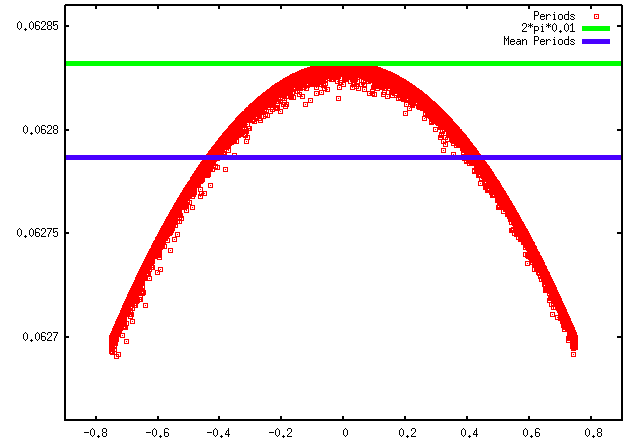} &
  \includegraphics[scale=0.3]{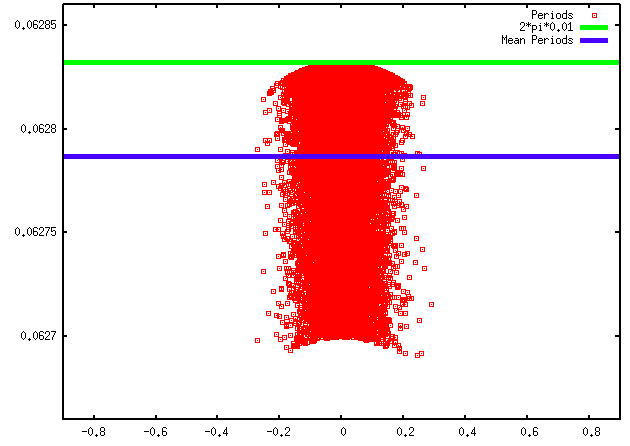}
\end{tabular}
\end{center}
\caption{In red the computed periods with respect to the initial positions (at
  left) and to the initial velocities (at right). The case of $E_{\eps}=-r^3$ with
  $\eps=0.01$. In blue the mean of the periods, in green the $2\pi\eps$ value.}
\label{DuffingBeamPeriodsFig}   
\end{figure}

In this section we consider the Vlasov equation \eqref{VlasAxSymBeam} provided
with the electric field given by \eqref{RappelForce} and with the initial
condition (see \cite{2007arXiv0710.3983F})
\begin{equation} \label{InitialConditionBeam}
  f_0(r,v)=\frac{n_0}{\sqrt{2\pi}v_{th}}\exp\left(-\frac{v^2}{2v_{th}^2}
  \right)\chi_{[-0.75,0.75]}(r),
\end{equation}
where the thermal velocity is $v_{th}=0.0727518214392$ and 
$\chi_{[-0.75,0.75]}(r)=1$ if $r\in[-0.75,0.75]$ and $0$ otherwise. In the
particle approximation, we implement this distribution function with $N_p=20000$
macroparticles and equal weights $w_i=1/N_p$. Moreover, this initial condition
will be used all along this work.

In this case, the characteristics equation or equivalently the Duffing equation
\eqref{DufDuf} has periodic solutions in time (see the appendix). Thus, the
time to make one rapid complete tour in the phase space depends only on the
initial condition. It is actually constant in time and it corresponds to the
period of the trajectory. The repartition of the periods with
respect to the initial condition in \eqref{InitialConditionBeam} is given in
Fig. \ref{DuffingBeamPeriodsFig}. These values are computed with a fine RK4
solver and they are in accordance with formula \eqref{2ndDL_T} in the appendix.
Notice that the periods are very close to $2\pi\eps$, they belong to
$(2\pi\eps-1.5\eps^2, 2\pi\eps)$.

Now we compare the modified ETD scheme with a reference solution, obtained by
solving \eqref{CharactCurvesOfGenVlas} with explicit 4th order Runge-Kutta scheme
with small time step. The solution in the phase space $(R,V)$ at different
large times, obtained with $\eps=0.01$, is illustrated in Fig.
\ref{DuffingCasePhaseSpace}.
\begin{figure}[ht]
\begin{center}\hspace*{-8mm}  
\begin{tabular}{cc}
  \includegraphics[scale=0.3]{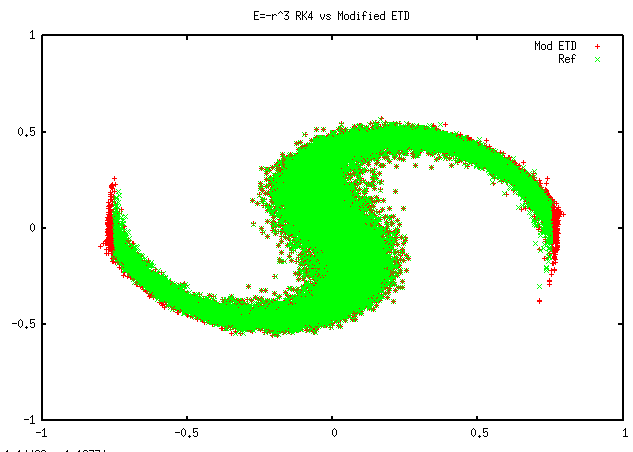} &
  \hspace*{-6mm}
  \includegraphics[scale=0.3]{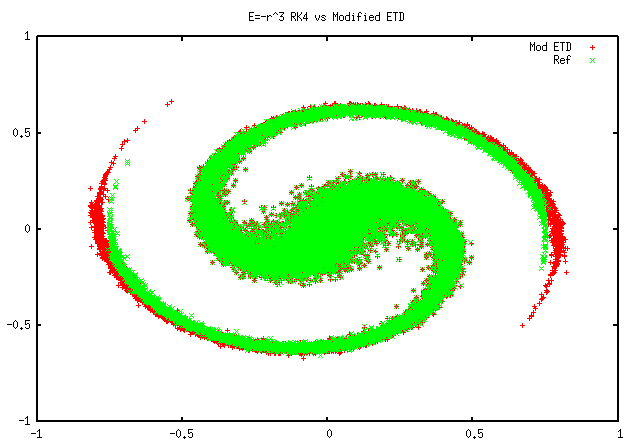}
  \\
    \includegraphics[scale=0.3]{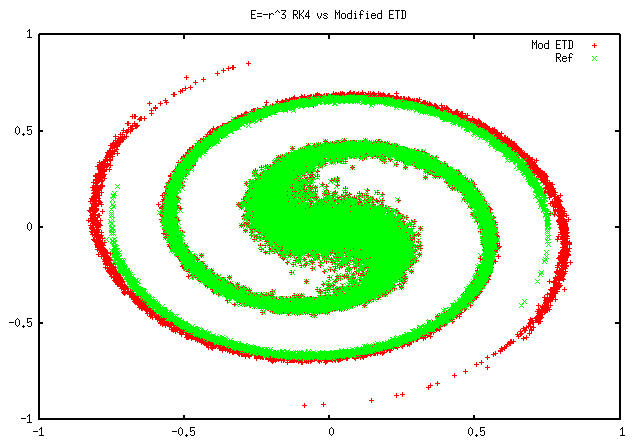} &
  \hspace*{-6mm}
  \includegraphics[scale=0.3]{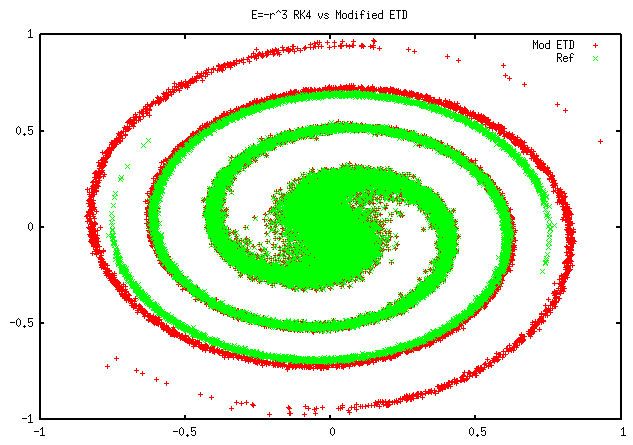}
\end{tabular}
\end{center}
\caption{ The case of $E_{\eps}=-r^3$ with $\eps=0.01$. In green: the reference
solution obtained with a Runge-Kutta algorithm with time step $\Delta t=2\pi\eps
/100$. In red: the result of the modified ETD scheme with time step 
$\Delta t=0.5$. The final times are, from left to right and from top to bottom $t=10$,
$20$, $30$ and $40$. }\label{DuffingCasePhaseSpace}  
\end{figure}

We notice that beyond $t=30$ the errors become significant. Next, we explain the
reasons. Imagining a position (or velocity) trajectory in time
being sinusoid-like, it is easy to see that a particle for which the period is
underestimated (by the mean period) will drift
inward in the phase space when applying the second step of the algorithm. When
the period is overestimated, the particle will clearly drift outward. In
addition, with the same approximation of the periods, larger is the amplitude of
the  sinusoid, bigger is the error made by the third step of the modified ETD
scheme (see Fig. \ref{ErrorsFuncAmp}). Consequently, all the particles close to
$(0,0)$ in the phase space are well treated by the algorithm, while the particles 
far from the center of the beam accumulate significant error in time. This
behaviour was noticed in reference \cite{frenod/hirstoaga/sonnendrucker} in short
time simulation, in terms of errors for particles off or on the slow manifold.
\begin{figure}[ht]
  \begin{center}\hspace*{-8mm}
    \begin{tabular}{cc}
      \includegraphics[scale=0.35]{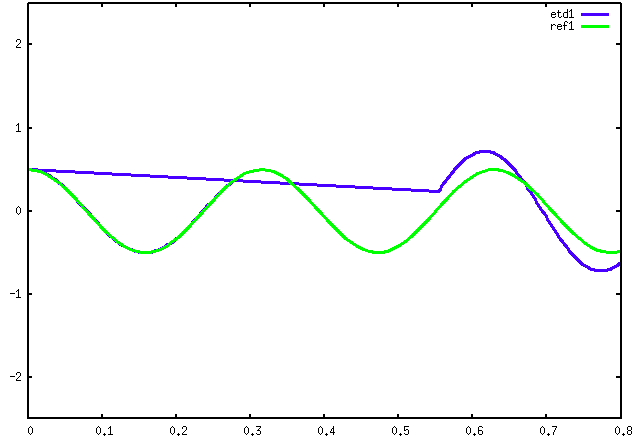} &
      \hspace*{-6mm}
      \includegraphics[scale=0.35]{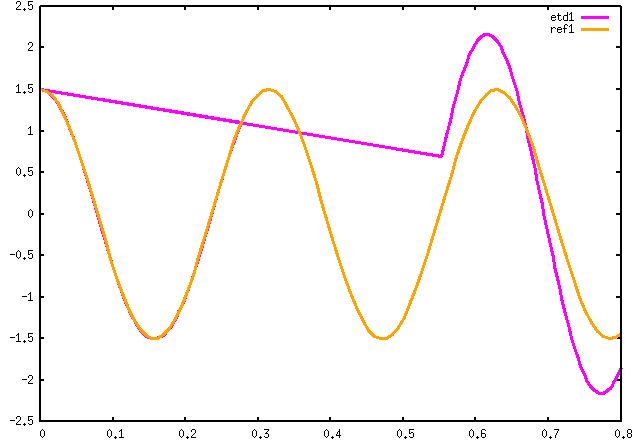}
    \end{tabular}
  \end{center}
  \caption{In green : the amplitude of the exact solution is $0.5$ and the period
    is $2\pi\eps$. In blue the solution obtained with the ETD Scheme with a
    period $2\pi\eps-15\eps^2$. In orange : the amplitude of the exact solution
    is $1.5$ and the period is $2\pi\eps$. In purple the solution obtained with
    the ETD Scheme with a period $2\pi\eps-15\eps^2$. } 
  \label{ErrorsFuncAmp}   
\end{figure}

\subsection{Application to long time simulation of the Vlasov-Poisson test case}
\label{sec2.3}

In this section we consider the Vlasov-Poisson equation
\eqref{VlasPoissAxSymBeam} provided with the initial condition
\eqref{InitialConditionBeam}. The calculus done in the appendix for computing
periods is now difficult to achieve. Therefore, in this case, we compute the
time to make one fast
tour by using a fine RK4 solver. The Poisson equation is simply solved by means
of the trapezoidal formula for the integral in $r$, with $256$ cells.
As in the previous section, the fast time depends on the initial condition but
in addition, it changes slowly in time. The same happens for the amplitude
$\left\Vert \yv \left(t\right)\right\Vert _{2}=\sqrt{R(t)^{2}+V(t)^{2}}$, see
Figs. \ref{VP_Periods_Rep_Evol} and \ref{VP_Modulus_Rep_Evol}.

The comparison between the modified ETD scheme and a reference solution obtained
with explicit 4th order Runge-Kutta scheme with small time step is summarized in
Fig. \ref{VPCasePhaseSpace}. We see from Figs. \ref{DuffingCasePhaseSpace} and
\ref{VPCasePhaseSpace} that the dynamics and/or the filamentation in the
Vlasov-Poisson case is slower than that in the case of $E_{\eps}$ given by
\eqref{RappelForce}. Certainly, this is because of bigger particle periods
(compare Figs. \ref{DuffingBeamPeriodsFig} and  \ref{Init_VP_Periods_Rep})
\begin{figure}[ht]
\begin{center}
  \includegraphics[scale=0.5]{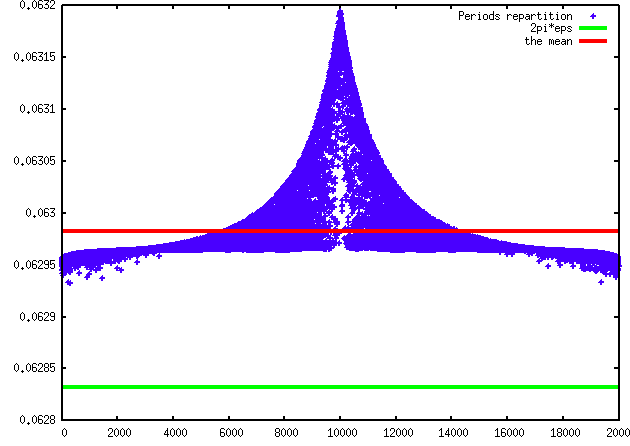}
\end{center}
\caption{In blue the computed periods with respect to the initial positions, in
the Vlasov-Poisson case with $\eps=0.01$. In red the mean of the periods, in green
the $2\pi\eps$ value.} \label{Init_VP_Periods_Rep}
\end{figure}
\begin{figure}[ht]
\begin{center}\hspace*{-8mm}  
\begin{tabular}{cc}
  \includegraphics[scale=0.3]{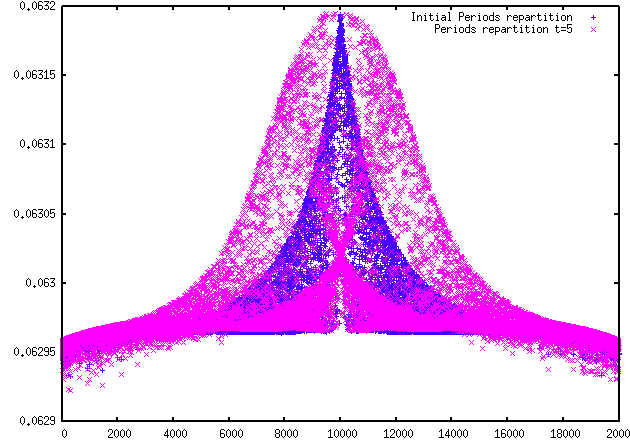} &
  \hspace*{-6mm}
  \includegraphics[scale=0.3]{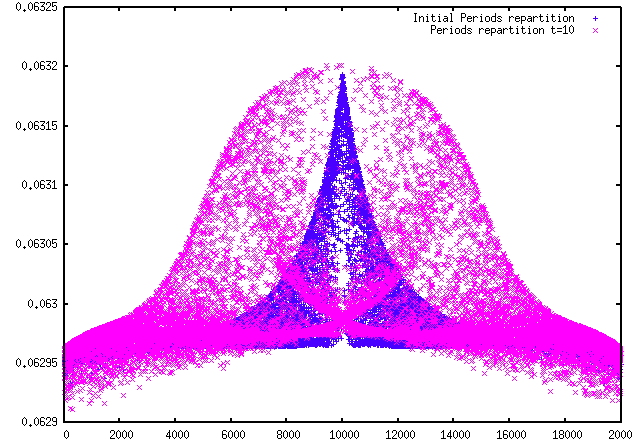}
  \\
    \includegraphics[scale=0.3]{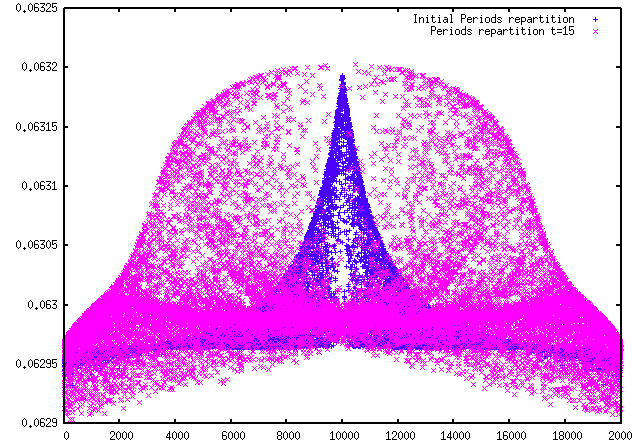} &
  \hspace*{-6mm}
  \includegraphics[scale=0.3]{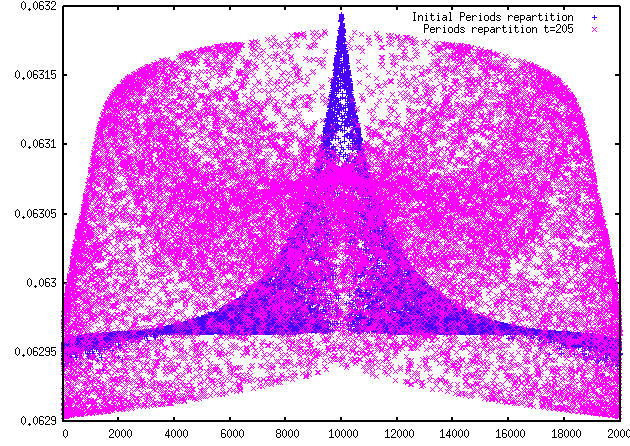}
\end{tabular}
\end{center}
\caption{The periods distribution with respect to the initial periods (in blue)
  at times $t=5, 10, 15$, and $20$. The Vlasov-Poisson case with $\eps=0.01$.}
\label{VP_Periods_Rep_Evol}  
\end{figure}
\begin{figure}[ht]
\begin{center}\hspace*{-8mm}  
\begin{tabular}{cc}
  \includegraphics[scale=0.3]{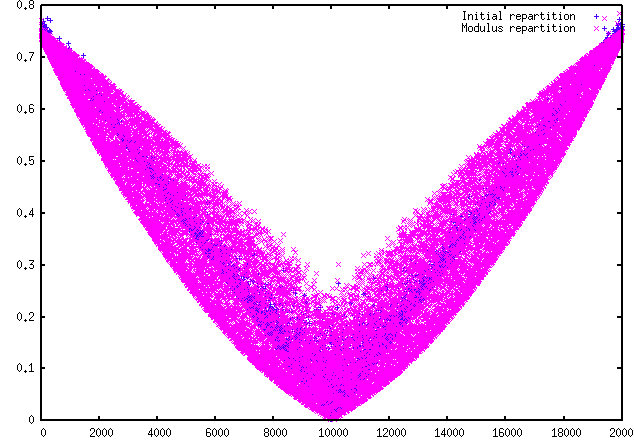} & 
  \hspace*{-6mm}
  \includegraphics[scale=0.3]{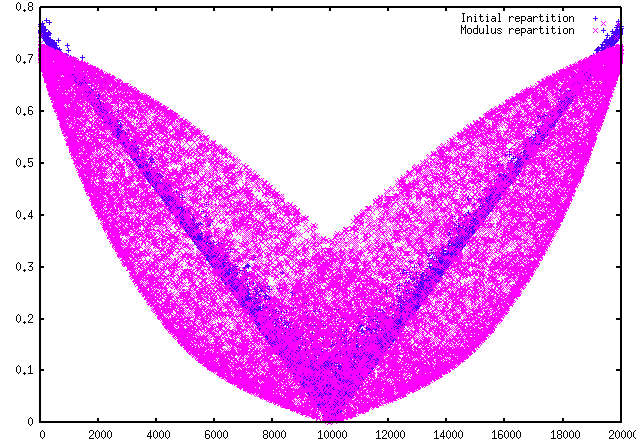} 
  \\
    \includegraphics[scale=0.3]{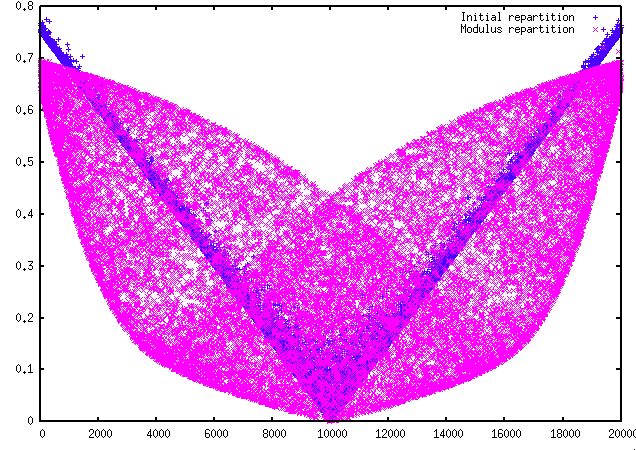} & 
  \hspace*{-6mm}
  \includegraphics[scale=0.3]{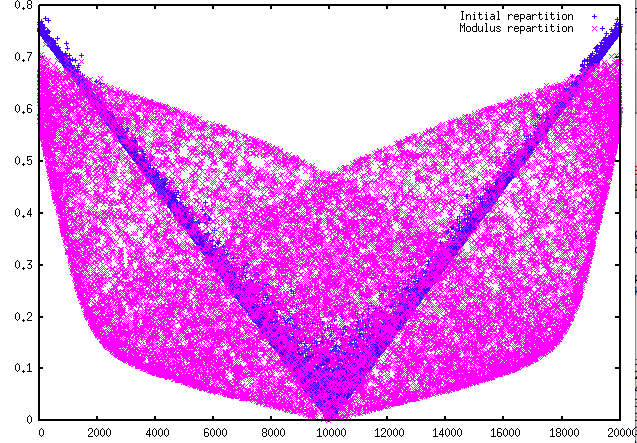} 
\end{tabular}
\end{center}
\caption{The modulus $\left\Vert \left(R\left(t\right),V\left(t\right)\right)
  \right \Vert _{2}$ with respect to the initial modulus (in blue) at times
  $t=5, 10, 15$, and $20$. The Vlasov-Poisson case with $\eps=0.01$. }
\label{VP_Modulus_Rep_Evol}  
\end{figure}
\begin{figure}[ht]
  \begin{center}\hspace*{-8mm}  
    \begin{tabular}{cc}
      \includegraphics[scale=0.3]{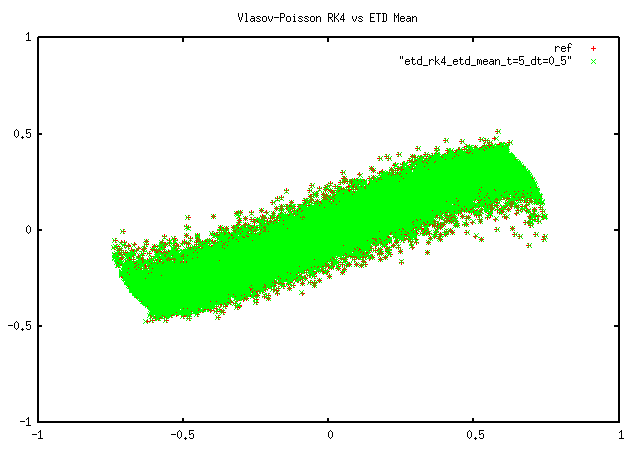} &
      \hspace*{-6mm}
      \includegraphics[scale=0.3]{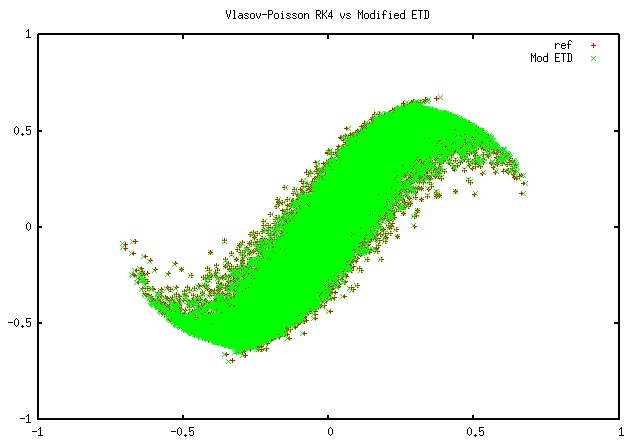} \\
    \includegraphics[scale=0.3]{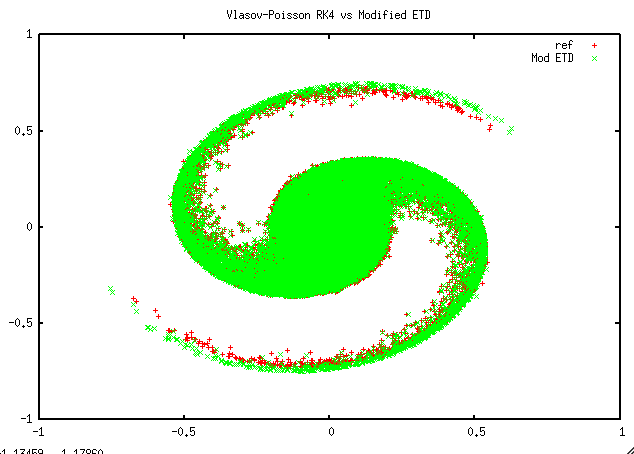} &
    \hspace*{-6mm}
    \includegraphics[scale=0.3]{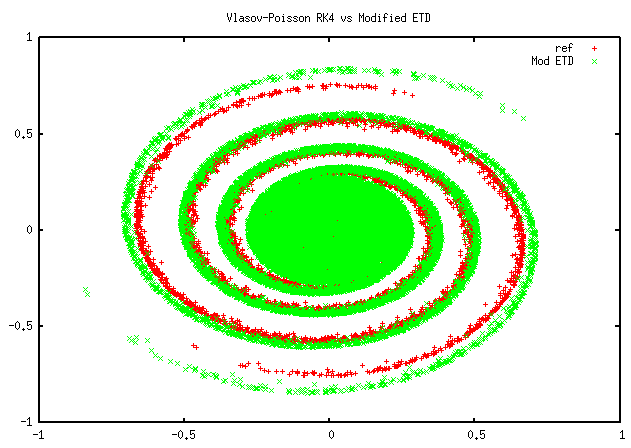}
    \end{tabular}
  \end{center}
  \caption{Vlasov-Poisson test case with $\eps=0.01$. In red: the reference
    solution obtained with a Runge-Kutta algorithm with time step $\Delta t=2\pi
    \eps/100$. In green: the result of the modified ETD scheme with time step
    $\Delta t=0.5$. The final times are, from left to right and from top to
    bottom: $t=5$, $10$, $30$ and $60$. }
  \label{VPCasePhaseSpace}  
\end{figure}

Now, we can recall the discussion from the previous section about the
consequences of using approximated period within the second and the third steps
of the algorithm. However, in this Vlasov-Poisson case, the periods distribution
is more complicated and thus, it seems more difficult to clearly state how the
approximation influences long time simulation. We only remark the following.
The error done for particles at the extremities of the beam is less significant
than in the previous Vlasov case. The reason is that 
the mean of the initial periods is very close to the periods of the particles
far from the center of the beam (see Figs. \ref{Init_VP_Periods_Rep} and
\ref{VP_Periods_Rep_Evol}). As for those close to the center $(0,0)$, we
can see in Fig. \ref{VPCasePhaseSpace} that particles slightly drifting both
outward and inward are present, as expected from Fig. \ref{Init_VP_Periods_Rep}.

Therefore, in the next section we will compute at each time step the period of
each particle, in order to take into account this slow evolution of the periods
distribution. We recall that the slow evolution of the amplitudes is considered
when applying the first step of the algorithm.

\section{The new time-stepping scheme}
\label{Section3}  





In this section, we change Algorithm \ref{ClassicETDAlgo} by improving the
approximation of each particle's time to make one rapid complete tour.
More precisely, within the first step, instead of solving the ODE
\eqref{CharactCurvesOfGenVlas} during some fixed fast time for all the particles,
we will push each particle
with its own fast time. Unfortunately, implementing this idea in Algorithm
\ref{ClassicETDAlgo} will not work in the Vlasov-Poisson case. The reason is the
following. Suppose that we have the particles $(\yv_i)_{\inp}$ and their computed
periods $(T_i)_{\inp}$ at time $t_n$. First, we finely solve the ODE with each
particle as initial condition, during different times (the $T_i$s), paying 
attention to push enough in time (the maximum period) all the
particles in order to calculate the correct self-consistent electric field.
Then, applying the second step of Algorithm \ref{ClassicETDAlgo} with the
corresponding $N_i$ will lead to particles at different times $t_n+N_i\,T_i$.
Now, it is impossible to apply the third step of the algorithm, because the
particles arrived at the minimum of $\big(t_n+N_i\,T_i\big)_{\inp}$ can not be
pushed with the right electric field since all the particles are not
available at that time. This problem is solved by changing the steps of the
algorithm, as it is described below, in Algorithm \ref{ETDFSSModAlgo}.

Firstly, note that dynamical system \eqref{CharactCurvesOfGenVlas} can be
rewritten as 
\begin{align}
   &\boldsymbol{y}'\left(t\right)=\frac{1}{\eps}M\boldsymbol{y}\left(t\right)
  +\boldsymbol{L}^{\eps}\left(t,\boldsymbol{y}\left(t\right)\right),
   \label{StiffODEForBeam}   
\end{align}
where $\boldsymbol{y}$ is defined by \eqref{SimplNotations} and $M$ and
$\boldsymbol{L}^{\eps}$ are given by
\begin{equation}
M=\left(\begin{array}{cc}
~0 & 1\\
-1 & 0\end{array}\right)
\quad \text{and} \quad
\boldsymbol{L}^{\eps}=\left(\begin{array}{c}
0\\
E_{\eps}\end{array}\right).
\end{equation}
Let $T$ be a positive real number which is intended to be close to $2\pi$ and
let $\boldsymbol{r}\left(\tau\right)$ be the matrix defined by
\begin{align}
  &\boldsymbol{r}\left(\tau\right)=\left(\begin{array}{cc}
      ~~\cos\big(\frac{2\pi}{T}\tau\big) & \sin\big(
      \frac{2\pi}{T}\tau\big)\\
      -\sin\big(\frac{2\pi}{T}\tau\big) & \cos\big(
      \frac{2\pi}{T}\tau\big)\end{array}\right),
  \label{TRotationDef}  
\end{align}
which is a $T$-periodic function. Then we compute
\begin{eqnarray}
  \frac{\text{d}}{\text{d}\tau}\left(\boldsymbol{r}\left(-\frac{\tau}{\eps}
  \right) \boldsymbol{y}\left(\tau\right)\right)
  &=& \frac{2\pi}{\eps T}\boldsymbol{r}\left(-\frac{\tau}{\eps}\right)M
  \boldsymbol{y}\left(\tau\right)+\boldsymbol{r}\left(
  -\frac{\tau}{\eps}\right)\left(\frac{1}{\eps}M\boldsymbol{y}\left(t\right)+
  \boldsymbol{L}^{\eps}\left(\tau,\boldsymbol{y}\left(\tau\right)\right)\right)
  \nonumber \\
   &=&\boldsymbol{r}\left(-\frac{\tau}{\eps}\right)\boldsymbol{L}^{\eps}\left(
  \tau,\boldsymbol{y}\left(\tau\right)\right)+\left(1-\frac{2\pi}{T}\right)
  \frac{1}{\eps}\boldsymbol{r}\left(-\frac{\tau}{\eps}\right)M\boldsymbol{y}
  \left(\tau\right) \nonumber  \\
   &=&\boldsymbol{r}\left(-\frac{\tau}{\eps}\right)\boldsymbol{\beta}^{\eps}
  \left(\tau,\boldsymbol{y}\left(\tau\right)\right),
\label{13Feb}  
\end{eqnarray}
where 
\begin{equation}  \label{DefBetatau}  
  \boldsymbol{\beta}^{\eps}\left(\tau,\boldsymbol{y}\left(\tau\right)\right)=
  \boldsymbol{L}^{\eps}\left(\tau,\boldsymbol{y}\left(\tau\right)\right)+\left(
  1-\frac{2\pi}{T}\right)\frac{1}{\eps}M\boldsymbol{y}\left(\tau\right).
\end{equation}
Integrating \eqref{13Feb} between $s$ and $t$ with $s<t$ thus leads to
\begin{align}
     \boldsymbol{y}\left(t\right)=\boldsymbol{r}\left(\frac{t-s}{\eps}\right)\boldsymbol{y}\left(s\right)+\boldsymbol{r}\left(\frac{t-s}{\eps}\right)\int_{s}^{t}\boldsymbol{r}\left(\frac{s-\tau}{\eps}\right)\boldsymbol{\beta}^{\eps}\left(\tau,\boldsymbol{y}\left(\tau\right)\right)d\tau.
     \label{ExactIntegralFormula}  
\end{align}
Now we establish the time-stepping scheme. We write equation
\eqref{ExactIntegralFormula} with $s=t_n$ and $t=t_{n+1}=t_n+\Delta t$ in order
to specify how the solution is computed at time $t_{n+1}$ from its known value at
time $t_n$. Thus, for each particle of the beam we are faced with the numerical
computation of the integral from $t_n$ to $t_{n+1}$ involved in the right hand
side of \eqref{ExactIntegralFormula}. Subsequently, we denote by $\yv_i$ the
particle satisfying equation \eqref{StiffODEForBeam} provided with the initial
condition $\yv_i^n$. We denote by $T_i^n$ the time to make the first rapid
complete tour starting at $t_n$. This time is computed numerically.\\

\noindent Since we want to build a scheme with a time step $\Delta t$ much larger
than the fast oscillation, we first need to find the unique positive integers
$N_i^n$ and the unique reals $o_i^n\in[0,T_i^n)$ such that
\begin{equation}\label{TimeStepDec}    
  \Delta t=N_i^nT_i^n+o_i^n.
\end{equation}
The derivation of the scheme, Algorithm \ref{ETDFSSModAlgo}, is based on the
following approximation.

\begin{approximation}
\label{MainApprox}
We have for any $\inp$
\begin{equation}\label{Approx1}
  \int_{t}^{t+N_i^nT_i^n}\boldsymbol{r}^n_{i}\left(\frac{t-\tau}{\eps}\right)
  \boldsymbol{\beta}^{\eps}\left(\tau,\boldsymbol{y}_{i}\left(\tau\right)\right)
  d\tau\;\approx\; N_{i}^{n}\int_{t}^{t+T_i^{n}}\boldsymbol{r}^n_{i}\left(
  \frac{t-\tau}{\eps}\right)\boldsymbol{\beta}^{\eps}\left(\tau,\yv_i\left(
  \tau\right)\right)d\tau
\end{equation}
where $t=t_n+o_i^n$ and $\rv_i^n$ corresponds to the matrix in
\eqref{TRotationDef} with $T=T_i^n/\eps$.
\end{approximation}

\begin{remark}
Approximation \ref{MainApprox} is valid if we make the assumptions that the
times for all particles to do the first rapid tour starting from $t$ 
evolve slowly in time and that the particle and the electric field evaluated at
the particle position are quasi-periodic in time
(with a period close to the time to make the first rapid complete tour).
\end{remark}
\begin{lemma}
\label{FirstStepOfTheAlgLemma}   
Under Approximation \ref{MainApprox} we obtain for any $\inp$ and any time $t_n$
\begin{equation}
  \yv_i(t_{n+1}) \approx  \boldsymbol{y}_i(t_n+o_i^n) +
  N_i^n \Big(\yv_i({t_n+o_i^n+T_i^n}) - \boldsymbol{y}_i(t_n+o_i^n) \Big),
  \label{MainFormula1}  
\end{equation}
where $\yv_i$ is the solution to equation \eqref{StiffODEForBeam} with a given
initial condition $\yv_i(t_n)=\yv_i^n$.
\end{lemma}

\begin{proof}
Writing the variation-of-constants formula from $s=t_n+o_i^n$ to $t=t_{n+1}$ and
recalling \eqref{TimeStepDec} we obtain the following exact scheme
\begin{equation}\label{thescheme}
\yv_i(t_{n+1})=\yv_i(t_n+o_i^n) + \int_{t_n+o_i^n}^{t_{n+1}}
\rv_i^n\left(\frac{t_n+o_i^n-\tau}{\eps}
\right)\,\boldsymbol{\beta}^{\eps}\left(\tau,\yv_i\left(\tau\right)\right)d\tau
\end{equation}
since $\rv_i^n(N_i^nT_i^n/\eps)$ is the identity matrix. Then, under Approximation
\ref{MainApprox} we obtain 
\begin{equation}\label{rel}
\yv_i(t_{n+1}) \approx \yv_i(t_n+o_i^n) + N_i^n \int_{t_n+o_i^n}^{t_n+o_i^n+T_i^{n}}
\rv_i^n \left( \frac{t_n+o_i^n-\tau}{\eps}\right)\boldsymbol{\beta}^{\eps}\left(
\tau, \yv_i\left( \tau\right)\right)d\tau.
\end{equation}
Applying now formula \eqref{ExactIntegralFormula} with $s=t_n+o_i^n$, $t=
t_n+o_i^n+T_i^n$ and $\rv= \rv_i^n$ yields
\begin{equation}
  \yv_{i}\left(t_{n}+o_{i}^{n}+T_{i}^{n}\right)-\yv_{i}\left(t_{n}+o_i^n\right)
  = \int_{t_n+o_i^n}^{t_n+o_i^n+T_i^{n}}
  \rv_i^n \left( \frac{t_n+o_i^n-\tau}{\eps}\right)\boldsymbol{\beta}^{\eps}\left(
  \tau, \yv_i\left( \tau\right)\right)d\tau
  \label{13Juillet2013F1}    
\end{equation}
since $\rv_i^n(T_i^n/\eps)$ is the identity matrix. Thus, injecting \eqref{13Juillet2013F1}
in \eqref{rel} leads to \eqref{MainFormula1}.
\end{proof}

Using Lemma \ref{FirstStepOfTheAlgLemma}, we deduce the following algorithm to
compute $\yv_i^{n+1}$ from $\yv_i^n$.
\begin{algorithm}
\label{ETDFSSModAlgo}
Assume that for each $\inp$,  $\yv_i^n$, the value of a solution of
\eqref{StiffODEForBeam} at time $t_n$, is given.
\begin{enumerate}
\item Compute from $(\yv_i^n)_{\inp}$ the periods $(T_i^n)_{\inp}$ at time $t_n$.
\item For each $\inp$ compute the unique positive integers $N_i^n$ and the unique
reals $o_i^n\in[0,T^n_i)$ such that
\begin{equation}
  \Delta t = N^{n}_{i} T^n_i+o^{n}_{i}.
\end{equation}
\item For each $\inp$ compute $\yv_i(t_n+o^n_i)$ and $\yv_i(t_n+o_i^n+T_i^n)$ by
using a fine Runge-Kutta solver with initial condition $\yv_i^n$.
\item For each $\inp$ compute an approximation of $\yv_i(t_{n+1})$ thanks to
\begin{equation}
  \yv_i^{n+1}=\yv_i(t_{n}+o^n_{i})+N_{i}\Big(\yv_i(t_{n}+o_i^n+T_i^n)-
  \yv_i(t_{n}+o_{i}^n)\Big).
\end{equation}
\begin{figure}[h]
\begin{center}
\begin{tabular}{cc}
\includegraphics[scale=0.5]{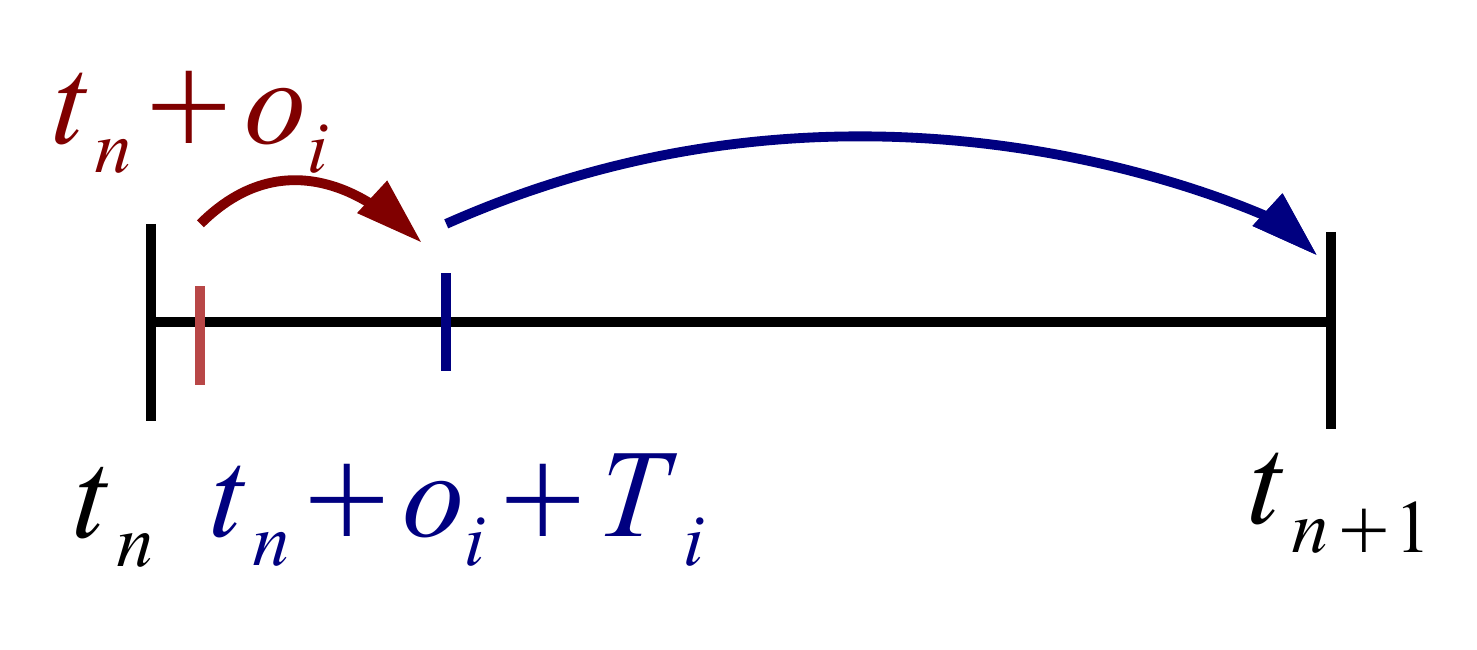} 
\end{tabular}
\end{center}
\end{figure}
\end{enumerate}
\end{algorithm}

\begin{remark}
\label{CompOfThePeriodsRem}  
As in \cite{frenod/hirstoaga/sonnendrucker}, the way to compute the $(T_i^n)_i$
from $(\yv_i^n)_i$ is the following. First, we solve the characteristics of
equation \eqref{CharactCurvesOfGenVlas} with a time step 
$(\Delta t)_{\!\text{RK}}=2\pi\eps/100$, for all initial conditions $\yv_i^n$ 
until all the particles reach their trajectory's third extremum. The criterion
for finding these extrema is the velocity's change of sign. We then state that
each particle's period is the time interval between the first and the third
extremum.
\end{remark}
\begin{remark}
\label{QuadraticInterpolationRem}
The time step $(\Delta t)_{\!\text{RK}}$ chosen in the third step of the
algorithm is $2\pi\eps/100$. It is easy to see that the reals $(t_n+o_i^n)_i$ and
$(t_n+o_i^n+T_i^n)_i$ are located on a few cells of length $(\Delta
t)_{\!\text{RK}}$. Thus, we approximate the values $\yv_i(t_n+o_i^n)$ and
$\yv_i(t_n+o_i^n+T_i^n)$ by quadratic interpolations.
\end{remark}


\section{Numerical simulations}
\label{Section4}  

We now do long time simulations with the {\em improved ETD scheme} (Algorithm
\ref{ETDFSSModAlgo}) in the two cases of Vlasov equation. The results show that
the new scheme works in long times while the modified ETD scheme becomes
unstable. We can observe the idea underlined from the begining of the paper:
using all along the simulation accurate approximations of each particle period
leads to very accurate solutions. Thus, we note that the numerical solutions in
the first case (Section \ref{sec4.1}) are very good since very precise values
for the periods are available and in addition they do not change in time. On the
contrary, in the Vlasov-Poisson case the periods are numerically approximated
at several levels and therefore the scheme does not perform as well as in the
first case.

\subsection{Long time simulation of the undamped and undriven Duffing equation}
\label{sec4.1}

In this section we consider the Vlasov equation \eqref{VlasAxSymBeam}
provided with the electric field given by \eqref{RappelForce} and with the
initial condition \eqref{InitialConditionBeam}. We recall that we use $20000$
macroparticles with equal weights.
\begin{figure}[ht]
\begin{center}\hspace*{-8mm}  
\begin{tabular}{cc}
  \includegraphics[scale=0.3]{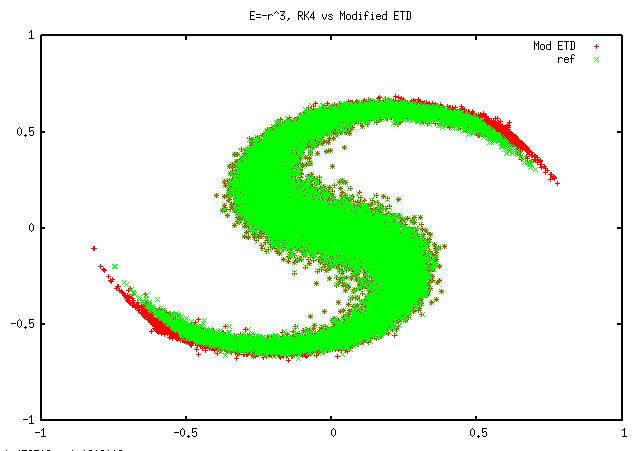} &
  \hspace*{-6mm}
  \includegraphics[scale=0.3]{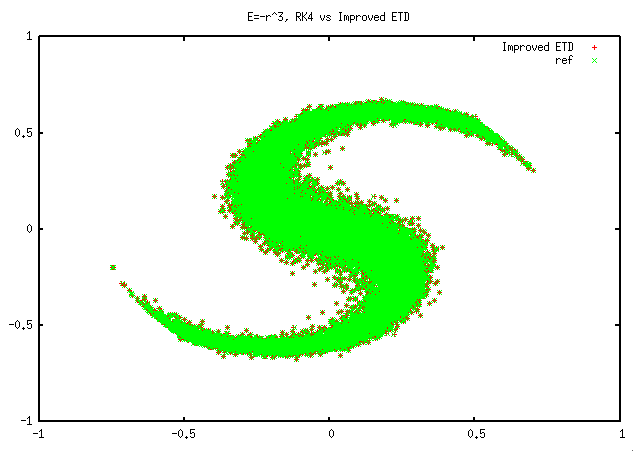}
  \\
  \includegraphics[scale=0.3]{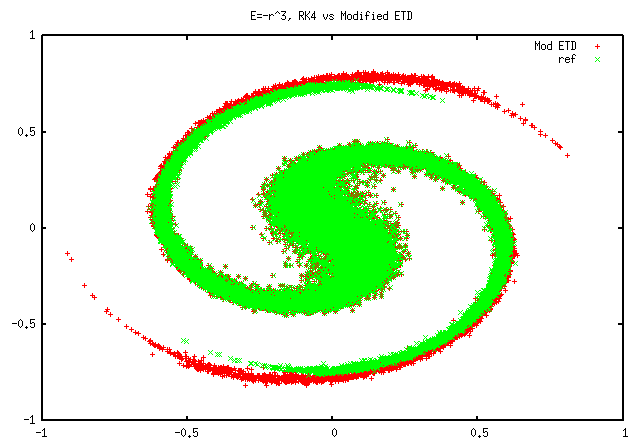} &
  \hspace*{-6mm}
  \includegraphics[scale=0.3]{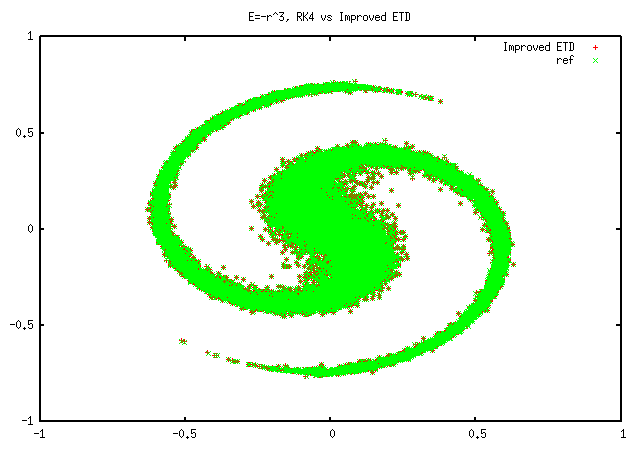}
  \\
    \includegraphics[scale=0.3]{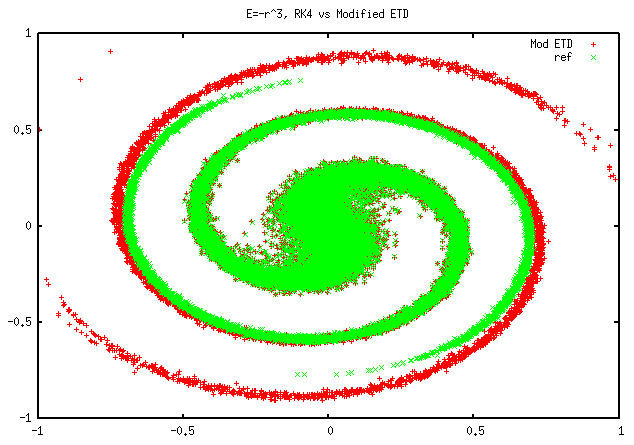} &
  \hspace*{-6mm}
  \includegraphics[scale=0.3]{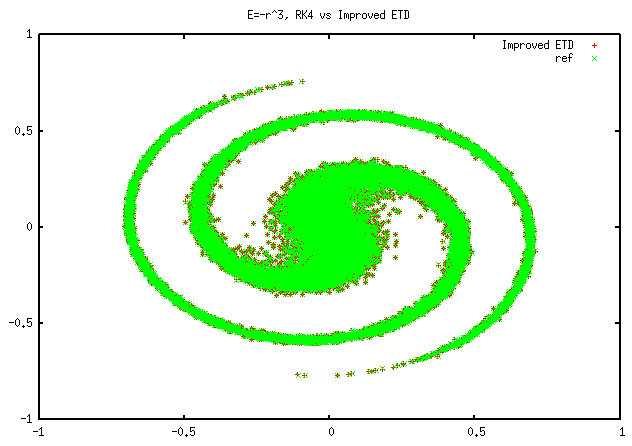}
    \\
  \includegraphics[scale=0.3]{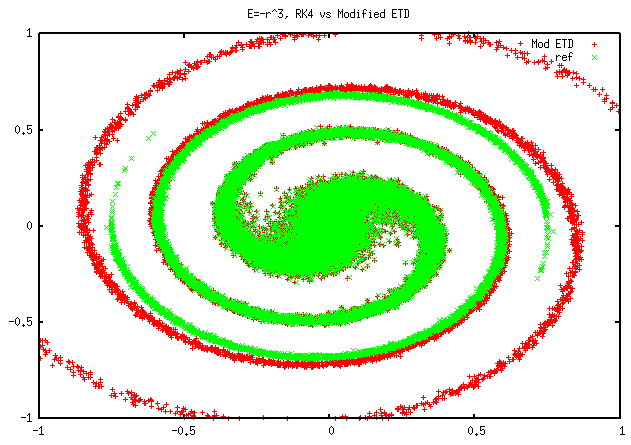} &
  \hspace*{-6mm}
  \includegraphics[scale=0.3]{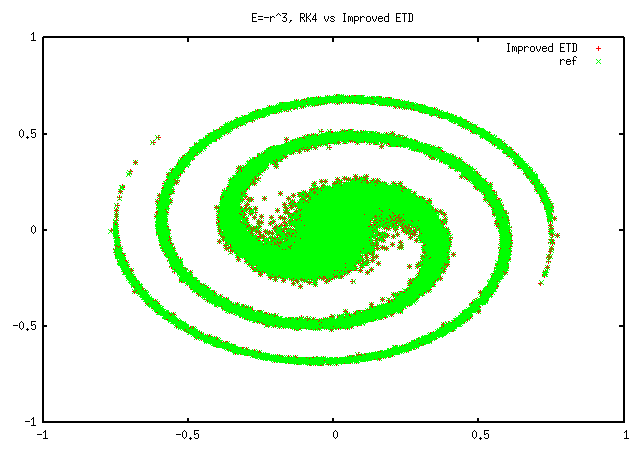}
\end{tabular}
\end{center}
\caption{The case of $E_{\eps}=-r^3$ with $\eps=0.001$. 
In green: the reference solution. In red: the solution obtained with the modified
ETD scheme (\textit{left}) and with the improved ETD scheme (\textit{right}).
From top to bottom : $t=10$, $t=20$, $t=30$ and $t=37$.
The time step used for the ETD schemes is $\Delta t=0.5$.}
\label{DuffCasePhaseSpaceLTImp}  
\end{figure}
Note that the first step of Algorithm \ref{ETDFSSModAlgo} is executed only
initially, since the particle periods are constant in time.

The results obtained with $\eps=0.001$ are summarized in Fig.
\ref{DuffCasePhaseSpaceLTImp} and with $\eps=0.0001$ in Fig. \ref{VPlast} at
right: we compare the outcome of the algorithm with a macroscopic time step
$\Delta t=0.5$ to a reference solution computed with a small time step,
$2\pi\eps/100$. As already mentioned in Section \ref{sec2.1} (see also the
appendix), the particles moving according to \eqref{CharactCurvesOfGenVlas}
with given $E_{\eps}$, have periodic trajectories and thus, the periods depend
only on the initial conditions. We observe on Figs.
\ref{DuffCasePhaseSpaceLTImp} and \ref{VPlast} that both the filamentation
phenomena and the fast rotation are well treated in long times by the
improved ETD scheme (the old scheme is unstable at these times).
Notice also that when $\eps=0.0001$, the integers $N^n_i$ are about $800$,
meaning that the ETD scheme is about $800$ times faster than the reference
solution.

\subsection{Long time simulation of the Vlasov-Poisson test case}
\label{LongTimeSimVPWithImprovedETDScheme}    

In this section we consider the Vlasov-Poisson equation \eqref{VlasPoissAxSymBeam}
provided with the initial condition \eqref{InitialConditionBeam}. The Poisson
equation is solved as mentioned in Section \ref{sec2.3}.
\begin{figure}[ht]
\begin{center}\hspace*{-8mm}  
\begin{tabular}{cc}
  \includegraphics[scale=0.3]{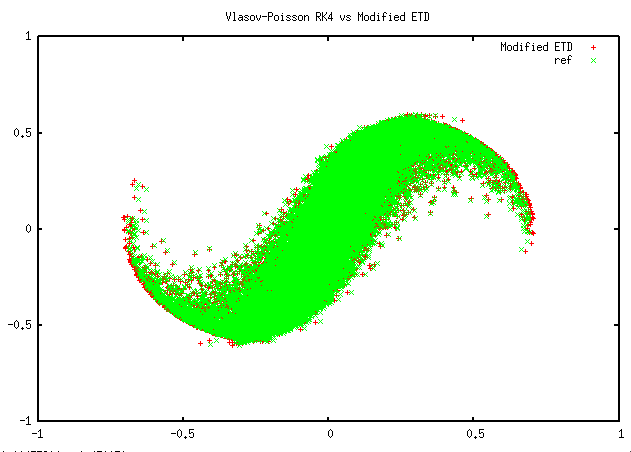} &
  \hspace*{-6mm}
  \includegraphics[scale=0.3]{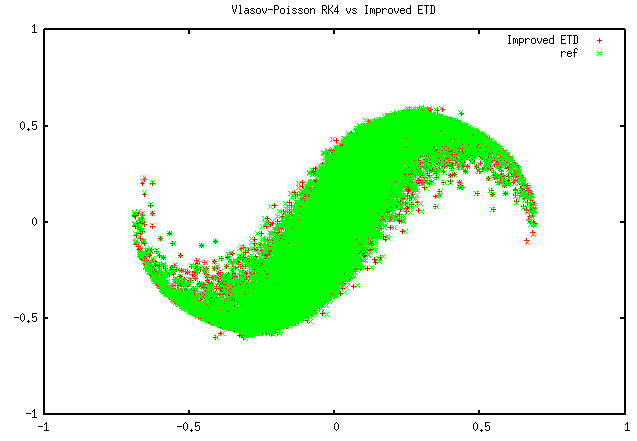} \\
  \includegraphics[scale=0.3]{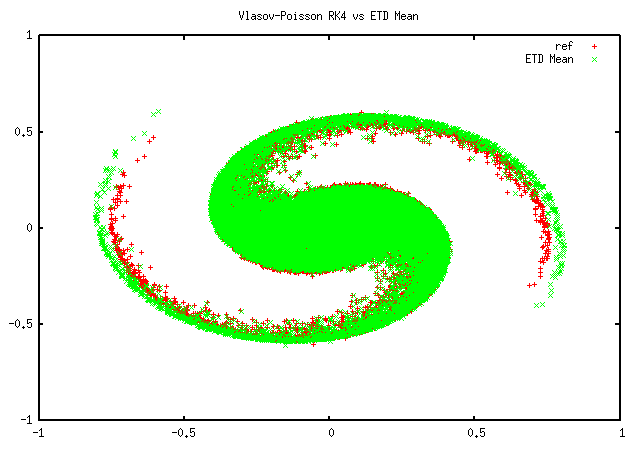} &
  \hspace*{-6mm}
  \includegraphics[scale=0.3]{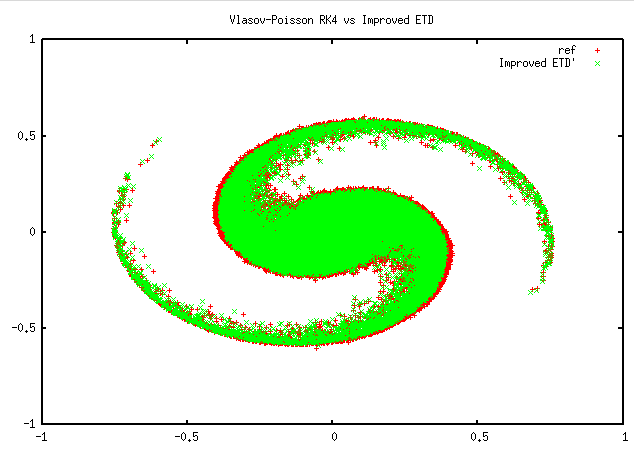} \\
  \includegraphics[scale=0.3]{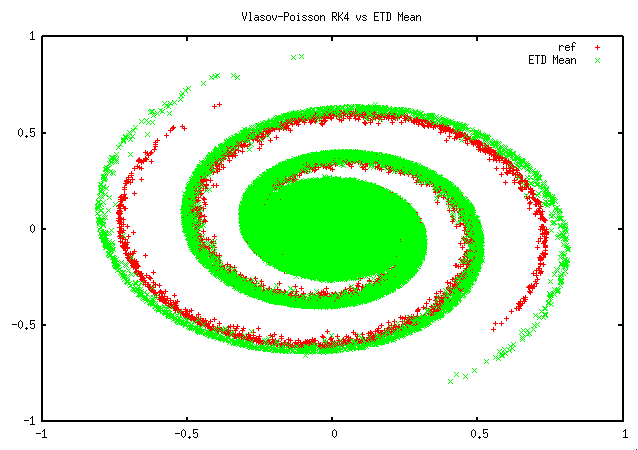} &
  \hspace*{-6mm}
  \includegraphics[scale=0.3]{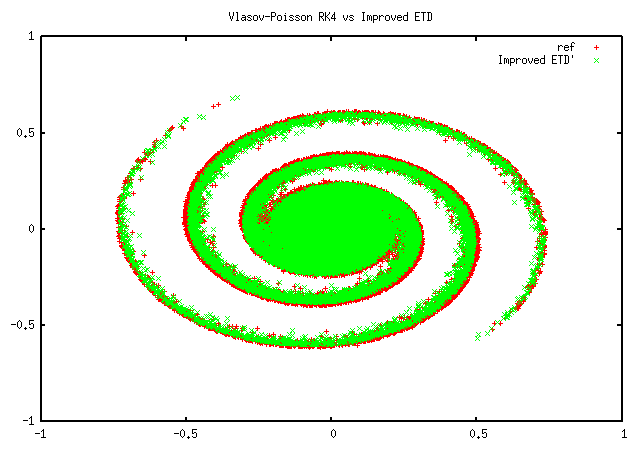} \\
    \includegraphics[scale=0.3]{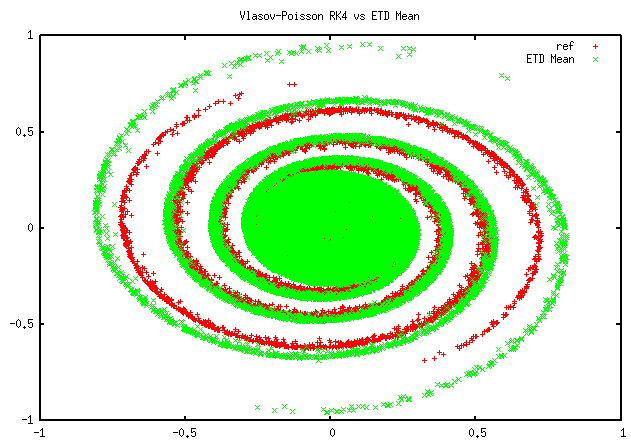} &
  \hspace*{-6mm}
  \includegraphics[scale=0.3]{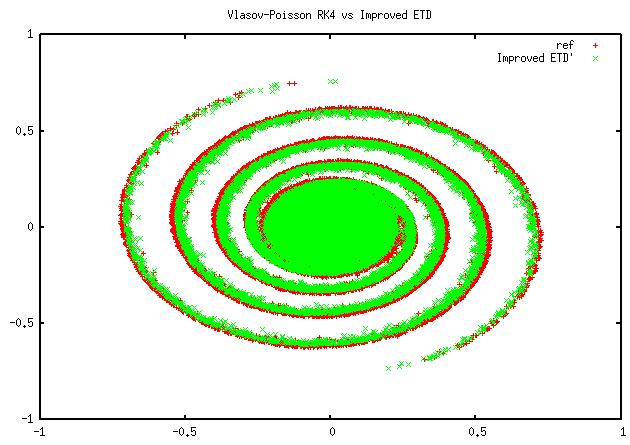}
\end{tabular}
\end{center}
\caption{Vlasov-Poisson test case with $\eps=0.001$. In red: the reference
solution. In green: the solution obtained with the modified
ETD scheme (\textit{left}) and with the improved ETD scheme (\textit{right}).
From top to bottom : $t=15$, $t=30$, $t=45$ and $t=60$. The time step used for
the ETD schemes is $\Delta t=0.5$.}
\label{VPCasePhaseSpaceLTImp}
\end{figure}
The results obtained with $\eps=0.001$ and the same time steps as above
are summarized in Fig. \ref{VPCasePhaseSpaceLTImp}. The case of
$\eps=0.0001$ is illustrated in Fig. \ref{VPlast} at left. Again, one can see
the slower filamentation of the Vlasov-Poisson case in comparison with
the given electric field case.
The time to make one rapid complete tour evolves slowly in time
(see Section \ref{sec2.3}). Therefore, the approximations made in Algorithm
\ref{ETDFSSModAlgo} lead to accurate results. In particular the particle
filaments are well treated by the improved ETD scheme. Note again that the
new algorithm is about $80$ times
faster (when $\eps=0.001$) than the reference solution.

Concluding, the scheme is able to produce quite accurate numerical solutions for
stiff Vlasov type equations for different small values of $\eps$ with the same
computational cost, which is rather close to that of a reduced model.

\begin{figure}[ht]
\begin{center}\hspace*{-8mm}
\begin{tabular}{cc}
\hspace*{3mm}
  \includegraphics[scale=0.33]{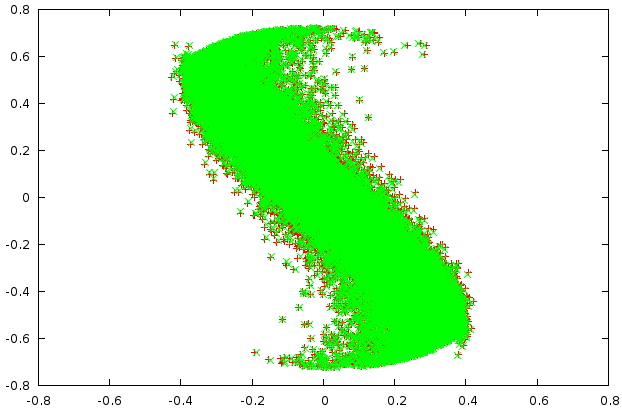} &
  \hspace*{-4mm}
  \includegraphics[scale=0.33]{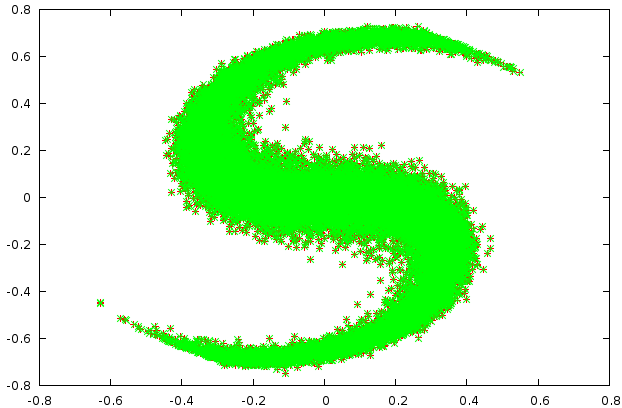} \\
\hspace*{3mm}
  \includegraphics[scale=0.33]{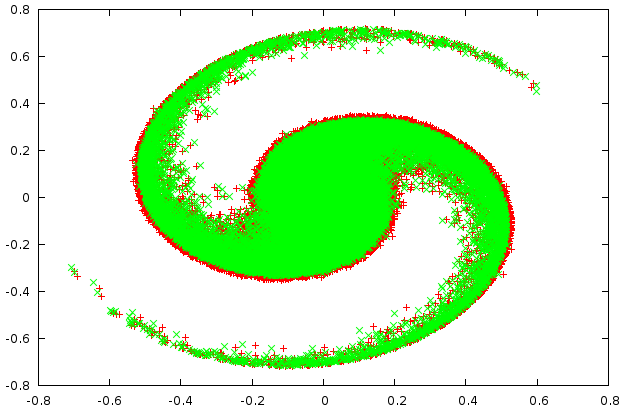} &
  \hspace*{-4mm}
  \includegraphics[scale=0.33]{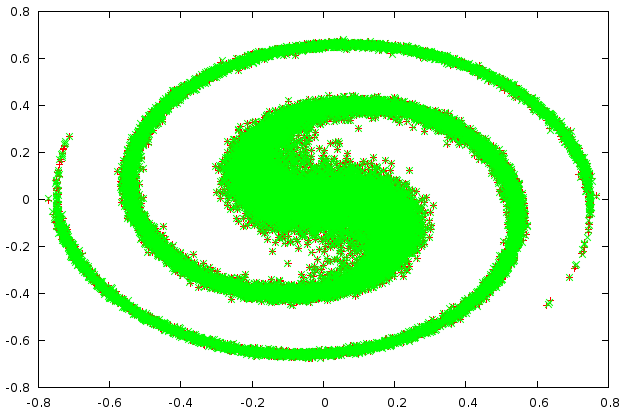} \\
\hspace*{3mm}
  \includegraphics[scale=0.33]{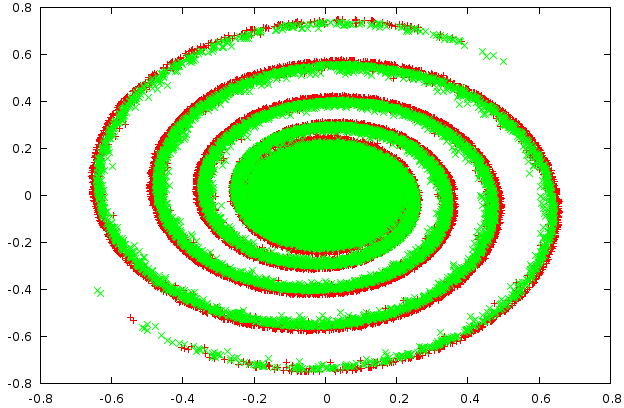} &
\hspace*{-4mm}
  \includegraphics[scale=0.33]{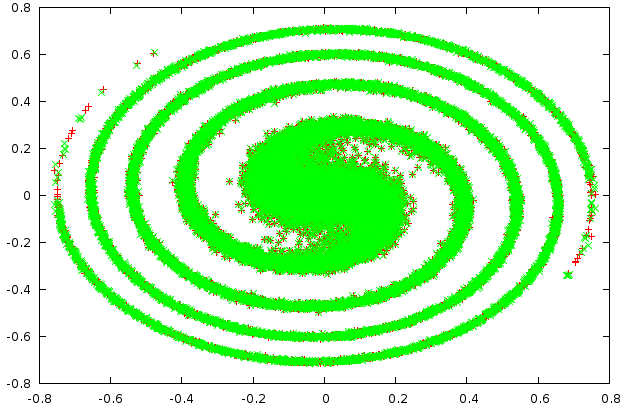} \\ 
\hspace*{3mm}
  \includegraphics[scale=0.33]{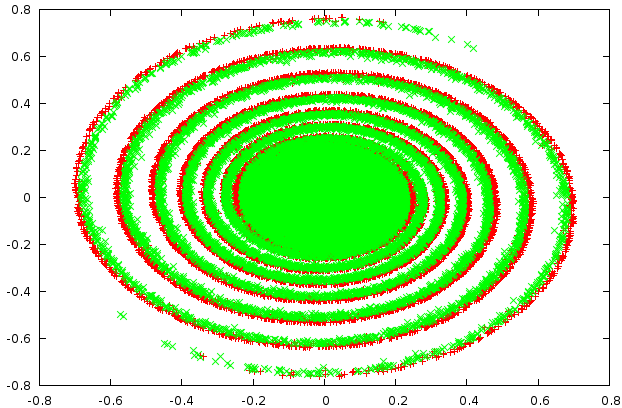} &
\hspace*{-4mm}
  \includegraphics[scale=0.33]{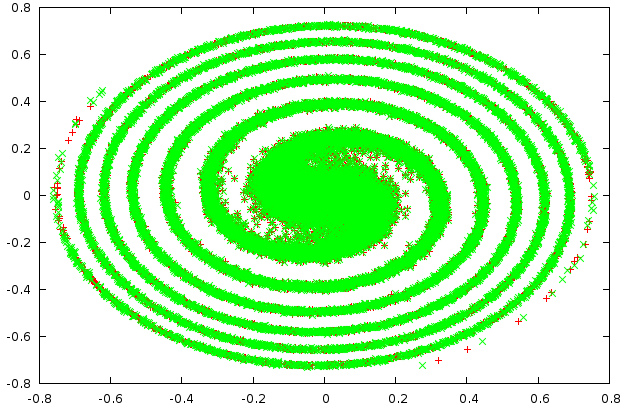} 
\end{tabular}
\end{center}
\caption{The cases of Vlasov-Poisson (\textit{left}) and of Vlasov with
$E_\eps = -r^3$ (\textit{right}), when $\eps=0.0001$. In red: the
reference solution. In green: the solution obtained with the improved ETD scheme.
From top to bottom : $t=10$, $t=30$, $t=60$ and $t=90$. The time step for
the ETD scheme is $\Delta t=0.5$.}
\label{VPlast}
\end{figure}


\appendix 
\section{Appendix: Periods calculus for an undamped and undriven Duffing equation}
\label{AnnexDuffing}   


\setcounter{equation}{0}

Let $\eps>0$ be some small parameter. The solution to the following Duffing-type
equation
\begin{gather}
\label{partDuf}
\left\{
\begin{aligned}
    & R''+ \frac{1}{\eps^2} R + \frac{1}{\eps} R^3=0 \\
    & R(0)=r_0, ~ R'(0)=v_0/\eps
\end{aligned}
\right.
\end{gather}
can be written by using Jacobi elliptic functions. In this appendix we obtain
by elementary calculus (see \cite{Hale}), an approximated value of order $3$ in
$\eps$ for the period of the solution. In particular, we obtain explicitly how
the period $T$ depend on the initial condition $(r_0,v_0)$.

The orbits of the solutions are level curves of the Hamiltonian function
\[
{\cal H}(R,V) = \frac{R^2+\eps^2V^2}{2\eps^2} + \frac{R^4}{4\eps},
\]
where $V$ stands for $R'$. As in Chapter V of reference \cite{Hale}, we
denote the restoring force associated to the hard spring described by equation
\eqref{partDuf} by 
\[
g(R) = \frac{1}{\eps^2} R + \frac{1}{\eps} R^3
\]
and by $G(R)=\cfrac{R^2}{2\eps^2} + \cfrac{R^4}{4\eps}$ the potential energy.
Since the function $G$ has an absolute minimum in $(0,0)$ and no other critical
point, all orbits are periodic. In addition, the closed orbit
of a periodic solution is symmetric with respect to the $R$-axis and intersect it
at two points, $(a,0)$ and $(b,0)$, with $a<b$. Taking into account that the
function $g$ is odd, we have $a=-b$ and finally, we can determine an implicit
formula for the period of the solution whose orbit intersect the $R$-axis in
$(b,0)$:
\begin{equation}\label{1st_eqT}
T = 4 \int_0^b\frac{1}{\sqrt{2h - 2G(R)}}dR,
\end{equation}
where $h$ is the value of the Hamiltonian on the solution of system
\eqref{partDuf} \begin{equation}\label{b_eq}
h = G(b) = {\cal H}(r_0,v_0).
\end{equation}
Replacing $h$ in \eqref{1st_eqT}, we obtain
\begin{equation}
T= 4\eps \int_0^1\frac{1}{\sqrt{1-R^2 + \frac{b^2\eps}{2}(1-R^4)}}dR.
\end{equation}
Then, solving equation \eqref{b_eq} for $b$ gives
\begin{equation}\label{b_val}
b=\sqrt{\frac{1}{\eps}\big(\sqrt{1+2\eps(r_0^2+v_0^2) +\eps^2r_0^4}-1\big)},
\end{equation}
which by Taylor expansion (recall $0<\eps\ll 1$) results in 
\begin{equation}\label{epsb^2_DL}
b^2\eps=(r_0^2+v_0^2)\eps+\frac{r_0^4-(r_0^2+v_0^2)^2}{2}\eps^2+{\cal O}(\eps^3).
\end{equation}
Therefore, the parameter $\alpha = \cfrac{b^2\eps}{2}$ is small when $\eps\ll 1$.
Thus, by Taylor expanding to second order the function
$$\alpha \mapsto \frac{1}{\sqrt{1-R^2 + \alpha (1-R^4)}}$$
for each $R$, we obtain the following approximation to the period 
\begin{equation}\label{1stDL_T}
T=2\pi\eps-\frac{3\pi}{4} b^2\eps^2 +\frac{57\pi}{128}b^4\eps^3+{\cal O}(\eps^4).
\end{equation}
By injecting \eqref{epsb^2_DL} in \eqref{1stDL_T}, we finally obtain 
\begin{eqnarray}\label{2ndDL_T}
T &=& 2\pi\eps-\frac{3\pi}{4}(r_0^2+v_0^2) \eps^2\nonumber\\
 & & + \left(\frac{105\pi}{128}(r_0^2+v_0^2)^2 - \frac{3\pi}{8}r_0^4\right)
\eps^3\nonumber\\
 & & + \, {\cal O}(\eps^4).
\end{eqnarray}

\bibliographystyle{plain}
\bibliography{biblio}

\end{document}